\documentclass{amsart}

\usepackage{amsmath, amssymb,amsthm}
\usepackage{hyperref}

\newcommand{\iaoi}{if and only if }

\newcommand{\mrm}{\mathrm}
\newcommand{\mcl}{\mathcal}

\newcommand{\R}{\text{$\mathbb{R}$}}
\newcommand{\N}{\text{$\mathbb{N}$}}

\newcommand{\FF}{\text{$\mathcal{F}$}}

\newcommand{\eps}{\text{$\varepsilon$}}
\newcommand{\ph}{\text{$\varphi$}}

\DeclareMathOperator*{\flim}{\mathcal{F}-lim}

\newcommand{\abs}[1]{\left\lvert#1\right\rvert}

\newcommand{\uip}[1]{\lceil #1 \rceil}

\newcounter{low}
\renewcommand{\thelow}{\rm (\alph{low})}
\newenvironment{lista}{\begin{list}{\thelow}{\usecounter{low}\setlength{\leftmargin}{3pc}\setlength{\itemsep}{0pc}}\setlength{\labelwidth}{10pc}}{\end{list}}

\newcommand{\refeq}[1]{{\rm (\ref{#1})}}
\numberwithin{equation}{section}

\newcommand{\dd}{\mrm{d}}
\newcommand{\limti}[1]{\lim\limits_{#1\to\infty}}
\newcommand{\emps}{\emptyset}
\newcommand{\sm}{\smallsetminus}

\newcommand{\ud}{\text{$\overline{d}$}}
\newcommand{\uud}{\text{$\overline{\overline{d}}$}}
\newcommand{\ld}{\text{$\underline{d}$}}
\newcommand{\lld}{\text{$\underline{\underline{d}}$}}
\newcommand{\uda}[1]{\text{$\overline{d_{#1}}$}}
\newcommand{\lda}[1]{\text{$\underline{d_{#1}}$}}

\theoremstyle{definition}
\newtheorem{definition}{Definition}[section]

\theoremstyle{plain}
\newtheorem{thm}[definition]{Theorem}
\newtheorem{lem}[definition]{Lemma}
\newtheorem{prop}[definition]{Proposition}
\newtheorem{cor}[definition]{Corollary}
\newtheorem*{prob}{Problem}

\theoremstyle{remark}
\newtheorem{rem}[definition]{Remark}

\begin{document}

\title{Range of density measures}
\author{Martin Sleziak}
\address{Department of Algebra, Geometry and Mathematical Education, Faculty of Mathematics, Physics and Informatics, Comenius University, Mlynsk\'a dolina, 842 48 Bratislava, Slovakia}
\email{\tt sleziak@fmph.uniba.sk}
\thanks{The first author was supported by VEGA grants 1/0519/09 and 1/0588/09}
\author{Milo\v{s} Ziman}

\begin{abstract}
We investigate some properties of density measures -- finitely
additive measures on the set of natural numbers $\N$ extending
asymptotic density. We introduce a class of density measures,
which is defined using cluster points of the sequence
$\big(\frac{A(n)}{n}\big)$ as well as cluster points of some other
similar sequences.

We obtain range of possible values of density measures for any subset of $\N$. Our description
of this range simplifies the description of Bhashkara Rao and Bhashkara Rao \cite{brbr}
for general finitely additive measures.
Also the values which can be attained by the measures defined
in the first part of the paper are studied.
\end{abstract}

\maketitle

\section*{Introduction}

We are interested in finitely additive measures defined on the algebra $\mcl{P}(\N)$
of all subsets of $\N$.
By a measure we mean a~function
$\mu:\, \mcl{P}(\N) \to [0,1]$ satisfying the following properties:
\begin{lista}
\item\label{dm0} $\mu(\N) = 1$;
\item\label{dm1} $\mu(A \cup B) = \mu(A) + \mu(B)$ for all disjoint $A, B \subseteq \N$.
\end{lista}

The \emph{asymptotic density} $d$ defined by $d(A) =
\lim\limits_{n \to \infty} \frac{A(n)}{n}$, where $A(n) = \big|A
\cap [1,n]\big|$, is a classical tool for measuring the size of
subsets of $\N$. But unfortunately, it is not defined on all
subsets of $\N$. Moreover, it is well known that the collection of
sets having asymptotic density (the domain of $d$) does not form
an algebra of sets. Let us denote this collection $\mcl{D}$.

Clearly, $\N \in \mcl{D}$ and $d(\N) = 1$.
If $A, B \in \mcl{D}$ and $A \cap B = \emptyset$, then $d(A \cup B) = d(A) + d(B)$.
Hence $d$ possesses both properties \ref{dm0} and \ref{dm1} above and
it is known that it is possible to extend $d$ to a measure.

We will study these extensions, i.e., the measures satisfying:
\begin{lista}
\setcounter{low}{2}
\item\label{dm2} $\mu|_{\mcl{D}} = d$.
\end{lista}
This kind of measure  will be called a \emph{density measure} (in
accordance with \cite{st}).

Existence of density measures was shown already by S.~Banach.
In functional analysis it is usually proved using Hahn-Banach
theorem (see e.g.~\cite[p.141,\S3]{BANACHLINEN}).
We will use a different approach for constructing density
measures, using ultrafilters (see e.g.~\cite[Theorem
8.33]{balste}, \cite[p.207]{hrjech}). Also the general theory of
extensions of a partial finitely additive measure to a measure,
described in detail by Bhashkara Rao and Bhashkara Rao in
\cite{brbr}, is a very convenient tool in this setting.

Dorothy Maharam \cite{maharam} pioneered the research of the
density measures on integers. This field was further studied
by Blass, Frankiewicz, Plebanek and Ryll--Nar\-dzew\-ski in \cite{bfpr}, van Douwen in
\cite{vandouwen} or \v Sal\'at and Tijdeman in \cite{st}. Recently
the density measures and related concept of L\'evy group
have been employed in the theory of social choice \cite{CAMPBELLKELLYASYMP, fey, lauwers, TOMATRADEOFF}.

Let us note that at least some form of axiom of choice is needed in the construction of finitely additive measures on $\N$,
since there exists a model of ZF constructed by Pincus and Solovay \cite{PINCUSSOLOVAY1977}
in which there are no nonprincipal finitely additive measures on $\N$, see also \cite{HOWARDRUBIN}.
(It was mistakenly stated in \cite{GREDENSSURV} that Buck's measure \cite{BUCK} yields
an effective construction of a density measure.)

\section{Expressions of density measures}

We start by describing the construction of density measures via
ultrafilters.

We first recall the notion of \emph{limit along a filter} (see
\cite[p.122, Definition 8.23]{balste}, \cite[p.206, Definition
2.7]{hrjech}). If $\FF$ is a filter on $\N$ and $(x_n)$ is a real
sequence then we say that $\flim x_n=L$ if $L$ is a real number
with the property
$$\{n; \abs{x_n-L}>\eps\} \in \FF$$
for each $\eps>0$.

We recall here some basic (and easy to show) properties of the
$\FF$-limit, which will be needed later.
\begin{lem}
Let $\FF$ be a free filter on $\N$ and $(x_n)$ be a real sequence.
\begin{enumerate}
\renewcommand{\theenumi}{\roman{enumi}}
\renewcommand{\labelenumi}{(\theenumi)}
  \item If $\limti n x_n=L$ then $\flim x_n=L$.
  \item If $\flim x_n$ exists, then $\liminf x_n \leq \flim x_n \leq \limsup x_n$.
  \item The $\FF$-limits are unique.
  \item $\flim (ax_n+by_n)=a\flim x_n+b\flim y_n$ $($provided the $\FF$-limits of $(x_n)$ and $(y_n)$ exist$)$.
  \item $\flim (x_n.y_n)=\flim x_n . \flim y_n$ $($provided the $\FF$-limits of $(x_n)$ and $(y_n)$ exist$)$.
  \item For every cluster point $c$ of the sequence $(x_n)$ there exists
    a free filter $\FF$ such that $\flim x_n=c$. On the other hand, if $\flim
    x_n$ exists, it is a cluster point of the sequence $(x_n)$.
  \item $\limti n x_n=L$ \iaoi $\flim x_n=L$ for every free ultrafilter $\FF$.
  \item If $(x_n)$ is bounded an $\FF$ is an ultrafilter, then $\flim x_n$ exists.
\end{enumerate}
\end{lem}

Using the above properties of $\FF$-limit one can show that for
any free ultrafilter $\FF$ on $\N$ a density measure $\mu_{\FF}$ can be
defined by
$$\mu_{\FF}(A)=\flim \frac{A(n)}n.$$
(We refer again to \cite[Theorem 8.33]{balste},
\cite[p.207]{hrjech} for the proof of this claim.)

A short notice of Lauwers in \cite{lauwers} claims:

{\em Every density measure can be expressed in the form
\begin{gather}\label{lauwers eq}
\mu_\ph(A) = \int_{\beta\N^*}\flim \frac{A(n)}{n}\, \dd\ph(\FF), \qquad A \subseteq \N
\end{gather}
for some probability Borel measure $\ph$ on the set of all
free ultrafilters $\beta\N^*$.}\\
But unfortunately our next considerations show that this result is
not correct.

\section{Density measures from $\alpha$-densities}

In this section we will consider another class of density
measures. In order to define them we need to recall the definition
of $\alpha$-densities.

For $\alpha \geq -1$ and $A \subseteq \N$ we denote
$A_\alpha(n) = \sum\limits_{k=1}^{n}\chi_A(k) k^\alpha$
and by $\mcl{D}_\alpha$ the set of all sets $A \subseteq \N$ such that
the sequence $\Big( \frac{A_\alpha(n)}{\N_\alpha(n)}\Big)$ has a limit. The limit of this
sequence we denote $d_\alpha(A)$ and we will call it the {\em $\alpha$-density} of the set $A$, i.e.,
$d_\alpha(A) = \lim\limits_{n \to \infty} \frac{A_\alpha(n)}{\N_\alpha(n)}$.
Hence for $\alpha = 0$ we have the asymptotic density and for $\alpha = -1$ the logarithmic
density.

As usual, by $\ld$ and $\ud$ we will denote the {\em lower\/}
and the {\em upper asymptotic density\/}, respectively, i.e.,
$\ld(A) = \liminf\limits_{n \to \infty} \frac{A(n)}{n}$ and
$\ud(A) = \limsup\limits_{n \to \infty} \frac{A(n)}{n}$. Similarly, we will call the functions
$\lda{\alpha}(A) = \liminf\limits_{n \to \infty} \frac{A_\alpha(n)}{\N_{\alpha }(n)}$
and $\uda{\alpha}(A) = \limsup\limits_{n \to \infty} \frac{A_\alpha(n)}{\N_{\alpha }(n)}$
the {\em lower\/} and the {\em upper $\alpha$-density\/}.

The following theorem is a consequence of the result of Fuchs and Giuliano Antonini in \cite{fga}.

\begin{thm}\label{fg thm}
Let $\alpha > -1$ and $f: \N \to \R$ be a bounded arithmetic function.
If
\begin{gather*}
\lim_{n \to \infty} \frac{1}{\N_\alpha(n)}\sum_{k=1}^{n}f(k)k^\alpha = L,
\end{gather*}
then
\begin{gather*}
\lim_{n \to \infty} \frac{1}{\N_\beta(n)}\sum_{k=1}^{n}f(k)k^\beta = L
\end{gather*}
for any $\beta \geq -1$.
\end{thm}

Replacing the function $f$ by the characteristic function of a set $A$ we get

\begin{cor}\label{alpha => beta}
If $A \in \mcl{D}_\alpha$ for some $\alpha > -1$, i.e., the $\alpha$-density
$d_\alpha(A) = \lim\limits_{n \to \infty} \frac{A_\alpha(n)}{\N_\alpha(n)}$
of a set $A$ exists, then $A \in \mcl{D}_\beta$ for all $\beta \geq -1$,
and $d_\alpha(A) = d_\beta(A)$.
\end{cor}

\begin{cor}
For all $\alpha > -1$ we have $\mcl{D}_\alpha = \mcl{D}$ and $d_\alpha = d$.
\end{cor}

This means that by replacing the sequence $\big(\frac{A(n)}{n}\big)$ in \refeq{lauwers eq} by
the sequence $\big(\frac{A_\alpha(n)}{\N_\alpha(n)}\big)$ for some $\alpha > -1$ we get a
density measure. By well-known inequality
$$\lda{-1} \leq \ld \leq \ud \leq \uda{-1}$$
(see \cite[p.241,Lemma V.2.1]{HALBERSTAMROTH}, \cite[p.272]{ten})  we get $\mcl{D}
\subseteq \mcl{D}_{-1}$ and $d_{-1}|_{\mcl{D}} = d$. Therefore
$d_{\alpha}$ is an extension of $d$ for $\alpha=-1$, too.

In particular, if we fix some $\alpha \geq -1$ and some free ultrafilter $\FF$, then the
mapping $A \mapsto \flim \frac{A_\alpha(n)}{\N_\alpha(n)}$ defines a density measure. Let us
denote this density measure by $\mu_{\alpha}^{\FF}$.

The following lemma can be useful for evaluating $\alpha$-densities.
\begin{lem}\label{alpha int}
For all $\alpha > -1$ we have
\begin{gather*}
\lim_{n \to \infty} \frac{n^{\alpha + 1}}{\N_\alpha(n)} = \alpha+1,
\end{gather*}
and for $\alpha = -1$
\begin{gather*}
\lim_{n \to \infty} \frac{\ln(n)}{\N_{-1}(n)} = 1.
\end{gather*}
\end{lem}

The routine proof can be done for example by interpreting the
sums appearing in the definition of $\N_\alpha(n)$ as the lower
and upper Riemann sums for integral of the function $x^\alpha$ or by using
Stolz theorem.

Now we will show that for every $\alpha > 0$, there is a free ultrafilter $\FF$
such that $\mu_\alpha^\FF$ is different from all density measures $\mu_\ph$ expressible
by \refeq{lauwers eq}.

Since the value of $\flim \frac{A(n)}{n}$ is a cluster point
of the sequence $\frac{A(n)}{n}$, we see that $\ld(A) \leq \flim \frac{A(n)}{n}
\leq \ud(A)$ for all free ultrafilters $\FF$ and all $A \subseteq \N$, and consequently
$\ld \leq \mu_\ph \leq \ud$ for every probability Borel measure $\ph$ on
$\beta\N^*$.

Finally we are ready to present a counterexample to the Lauwers' assertion:
Let $A = \bigcup\limits_{k=0}^\infty \big(2^{2k}, 2^{2k + 1}\big] \cap \N$.
Similarly as in Lemma \ref{alpha int} one can show that:
\begin{gather*}
\begin{aligned}
\lda\alpha(A) &= \lim\limits_{k \to \infty}\frac{A_\alpha(2^{2k})}{\N_\alpha(2^{2k})} =
\lim_{k \to \infty}\frac{\sum_{i=0}^{k-1}\int_{2^{2i}}^{2^{2i + 1}} x^\alpha\, \dd x}
{\frac{(2^{2k})^{\alpha + 1}}{\alpha +1}}
=\frac{2^{\alpha+1} - 1}{2^{2(\alpha+1)}-1} = \frac{1}{2^{\alpha+1}+1}, \\
\uda\alpha(A) &= \lim\limits_{k \to \infty}\frac{A_\alpha(2^{2k+1})}{\N_\alpha(2^{2k+1})} =
\lim_{k \to \infty}\frac{\sum_{i=0}^{k}\int_{2^{2i}}^{2^{2i + 1}} x^\alpha\, \dd x}
{\frac{(2^{2k+1})^{\alpha + 1}}{\alpha +1}} =
\frac{2^{\alpha + 1}(2^{\alpha + 1} -1)}{2^{2(\alpha + 1)} - 1}\\
& = \frac{2^{\alpha + 1}}{2^{\alpha + 1} + 1}
\end{aligned}
\end{gather*}
for all $\alpha > -1$. Hence $\lda\beta(A) < \lda\alpha(A) < \uda\alpha(A) < \uda\beta(A)$, if
$-1 < \alpha < \beta$.

 Now, taking any free ultrafilter
$\FF$ containing the set $\{2^{2k};\, k = 0,1, \dots\}$ we get:
\begin{gather*}
\mu_\alpha^\FF(A) = \flim \frac{A_\alpha(n)}{\N_\alpha(n)} = \lda\alpha(A).
\end{gather*}

This shows that the measure $\mu_\alpha^\FF$ cannot be of the form \refeq{lauwers eq},
if $\alpha > 0$.

Apart from providing a counterexample to \refeq{lauwers eq}, this
answers also one part of \cite[Question 7A.1]{vandouwen}. Van
Douwen asks, whether $\mu(A)\leq \ud(A)$ for every density
measure. The above procedure yields a measure $\mu_\alpha^\FF$
with $\mu_\alpha^\FF(A)=\uda\alpha(A)>\ud(A)$ (for $\alpha>0$ and
appropriate choice of the free filter $\FF$). A different example, based
on results of Bl\"umlinger \cite{blumlinger}, was presented in \cite{SLEZZIMDENSLEVY}.

Our previous observations lead to a more general class of density
measures than the one defined by Lauwers.

If a measure $\mu$ can be expressed in the form
\begin{gather}\label{new repr}
\mu(A) = \int_{\Omega} \mu_\alpha^\FF(A)\, \dd\psi(\FF, \alpha), \qquad A \subseteq \N
\end{gather}
for some probability Borel measure $\psi$ on the set $\Omega = \beta\N^* \times [-1, \infty)$,
then $\mu$ is a~density measure.

To be precise, we should check the existence of the integral in \refeq{new repr}.
As the function $f(\FF, \alpha) = \mu_\alpha^\FF(A) =
\flim \frac{A_\alpha(n)}{\N_\alpha(n)}$ is bounded, it suffices to show that
it is $\psi$-measurable for every Borel measure $\psi$. By Johnson \cite{johnson} a sufficient condition
for $f(\FF, \alpha)$ to be measurable is its separate continuity, i.e., continuity in $\FF$ for any
fixed $\alpha$ and continuity in $\alpha$ for any fixed $\FF$.

The continuity in $\FF$ follows immediately from the general theory of the
Stone-\v{C}ech compactification of a topological space (see e.g. \cite{gjrcf} or \cite{walker}).

For $\alpha > -1$, the continuity in $\alpha$ follows from the
estimations of Giuliano Antonini, Grekos and Mi\v{s}\'{\i}k
\cite{ggm}:
\begin{gather*}
\begin{aligned}
\limsup_{n \to \infty}\abs{\frac{A_\alpha(n)}{\N_\alpha(n)} -
\frac{A_{\alpha + \delta}(n)}{\N_{\alpha + \delta}(n)}} &< \frac{2 \delta}{\alpha + 1}\\
\limsup_{n \to \infty}\abs{\frac{A_\alpha(n)}{\N_\alpha(n)} -
\frac{A_{\alpha - \delta}(n)}{\N_{\alpha - \delta}(n)}} &< \frac{2 \delta}{\alpha - \delta + 1}
\end{aligned}
\end{gather*}
for $0 < \delta < \alpha + 1$.

It is proved in \cite{ggm} that there exists a set $A$ such that
the function $\lda{\alpha}(A)$ is discontinuous at $\alpha = -1$.
Thus our function $f(\FF, \alpha)$ cannot be continuous at $\alpha = -1$ for all filters $\FF\in\N^*$.
Hence we get the separate continuity on $\beta \N^* \times  (-1, \infty)$, only.

So $f$  is $\psi$-measurable on  $\beta \N^* \times  (-1, \infty)$ and on the measurable
set $\beta \N^* \times \{-1\}$ it is continuous, and thus Borel measurable. It follows
that $f$ is measurable on $\Omega$.

\section{Values of density measures}

Assume that $\mu$ is a density measure. Let $A \subseteq \N$. The
question is: {\em What are the possible values of $\mu(A)$.} Or:
{\em Which values can be attained by all density measures for a
fixed set $A$?} This question was proposed by Mark Fey in
\cite{fey2}.

It is clear that if $A \in \mcl D$, then $\mu(A) = d(A)$ for all
density measures $\mu$. But if it is not the case, there are more
possibilities for the value of $\mu(A)$. The next paragraphs
answer the above question.

The first estimation of $\mu(A)$ can be made using monotonicity of a measure.

If $B \subseteq A$ and $B \in \mcl D$, then $d(B) = \mu(B) \leq \mu(A)$. Hence
$\sup\{d(B);\, B\subseteq A,\, B\in \mcl{D}\} \leq \mu(A)$. Similarly,
$\inf\{d(C);\, C\supseteq A,\, C\in \mcl{D}\} \geq \mu(A)$. Let us denote
\begin{gather*}
\begin{aligned}
\lld(A) &= \sup\{d(B);\, B\subseteq A,\, B\in \mcl{D}\},\\
\uud(A) &= \inf\{d(C);\, C\supseteq A,\, C \in \mcl{D}\}.
\end{aligned}
\end{gather*}
Thus we get
\begin{thm}\label{first_estim}
For every set $A \subseteq \N$ and all density measures $\mu$ we have:
\begin{gather}
\lld(A) \leq \mu(A) \leq \uud(A).
\end{gather}
\end{thm}

Later on we will show that this estimation is the best possible.

Let us take
$$
d_*(A) = \sup \frac{\sum_{i=1}^{p}d(A_i) - \sum_{j=1}^{q}d(B_j)}{k}.
$$
The supremum is taken over all finite collections
$A_1, A_2, \dots, A_p$, $B_1, B_2, \dots B_q$ of sets in $\mcl D$ and positive integers $k$ such that
$$
k \chi_A + \sum_{j=1}^{q} \chi_{B_j}  \geq  \sum_{i=1}^{p} \chi_{A_i}.
$$
Similarly,
$$
d^*(A) = \inf \frac{\sum_{i=1}^{p}d(A_i) - \sum_{j=1}^{q}d(B_j)}{k}.
$$
The infimum is taken over all finite collections
$A_1, A_2, \dots, A_p$, $B_1, B_2, \dots B_q \in \mcl D$ and positive integers $k$ such that
$$
k \chi_A + \sum_{j=1}^{q} \chi_{B_j}  \leq  \sum_{i=1}^{p} \chi_{A_i}
$$

It is clear that
\begin{gather}\label{lld<di}
\lld(A)  \leq d_*(A) \leq d^*(A) \leq \uud(A).
\end{gather}

By Bhashkara Rao \cite[Theorem 3.2.9]{brbr} for every set
$A \subseteq \N$ and any value $x \in [d_*(A), d^*(A)]$ there is a
density measure $\mu$ such that $\mu(A) = x$. Moreover, if $\mu$
is a~density measure, then $\mu(A) \in [d_*(A), d^*(A)]$.

The definition of $d_*$ and $d^*$ (the range of density measures) is rather complicated.
The original result in \cite{brbr} was formulated for more general situation
of extending arbitrary partial measures. (Roughly said, by a partial measure we mean a
restriction of a measure to some class of subsets of $\N$. For more details we refer
the reader to \cite[Section 3.2]{brbr}.)
In our case, we work only with the asymptotic density and our aim is to prove the simplification
of this result. This simplification is contained in the following theorem and its corollary.

\begin{thm}\label{lld=di}
For every $A \subseteq \N$ the following is true:
\begin{gather}
\lld(A) = d_*(A)\qquad \text{and} \qquad \uud(A) = d^*(A).
\end{gather}
\end{thm}

\begin{cor}\label{dmvalues}
Let $A \subseteq \N$. There exists a density measure $\mu$ such that
$\mu(A) = x$ if and only if $x \in [\lld(A), \uud(A)]$.
\end{cor}

Let us note that if a partial measure $m$ is defined on an algebra of sets then
by a result due to \L{o}\'s and Marczewski \cite[Proposition 3.3.1]{brbr}
$\underline{\underline m}=m_*$ and $\overline{\overline{m}}=m^*$ holds for
this measure. This result cannot be used here, since $\mcl{D}$ is not closed
under intersections and unions. (In fact, the smallest algebra containing
$\mcl{D}$ is the whole powerset $\mcl{P}(\N)$.)

As we show in Remark \ref{REMPOLYA}, Theorem \ref{lld=di} and Corollary \ref{dmvalues}
could be deduced from results of P\'olya \cite{polya} using some functional analytic considerations.
However, we still find our proof of interest, since it is relatively elementary
and it is an interesting application of known results on density sets obtained by Grekos and Volkmann \cite{grv}.

Before we prove Theorem \ref{lld=di} we will describe some basic properties of $\lld$ and $\uud$.
The following lemma is crucial for proving some of them.
Let us note that the proof was inspired by the proof of \cite[Lemma 1]{st}.

\begin{lem}\label{4:LMM1}
If $A,B\in\mcl D$, $d(A)<d(B)$, then there exists $D\in\mcl D$ such that
$A\cap B\subseteq D\subseteq B$ and $d(D)=d(A)$.
\end{lem}

\begin{proof}
Put $C = A\cap B$, $A' = A\smallsetminus C$, $B' = B\smallsetminus C$.
We have $\lim\limits_{n \to \infty} \frac{C(n)+B'(n)}n=d(B)$
and $\lim\limits_{n \to \infty} \frac{C(n)+A'(n)}n=d(A)$, hence
\begin{gather*}
L = \lim\limits_{n \to \infty} \frac{B'(n)-A'(n)}n=d(B)-d(A)>0.
\end{gather*}
We shall construct a subset $D'\subset B'$ such that $\lim\limits_{n \to \infty}
\frac{D'(n)-A'(n)}n = 0$. Then for $D = C\cup D'$ we have
$d(D) = \lim\limits_{n \to \infty} \frac{C(n)+D'(n)}n =
\lim\limits_{n \to \infty} \frac{C(n)+A'(n)}n +
\lim\limits_{n \to \infty} \frac{D'(n)-A'(n)}n=d(A)$, so $D$ is the desired subset of $B$.

The subset $D'$ is defined by induction. If $n\notin B'$, then $n\notin
D'$. If $n\in B'$ and $D'(n-1)+1>A'(n)$, then $n\notin D'$. If $n\in B'$
and $D'(n-1)+1 \leq A'(n)$, then $n\in D'$. It is obvious that $D'(n) \leq A'(n)$.

Let us note that if $m\in B'$ but $m\notin D'$ (the second case), then
$D'(m)+1>A'(m)\geq D'(m)$, hence $D'(m)=A'(m)$. If $n\in\N$ and $m$ is the
largest number such that $m\leq n$, $m\notin D'$ and $m\in B'$, then for
every $k$, $m < k \leq n$, we have $D'(k)-D'(m)=B'(k)-B'(m)$ (since all
members of $B'$ in the interval $(m,n]$ belong to $D'$). This implies
$B'(k)-B'(m) \leq A'(k)-A'(m)$.

We denote the largest number $m\leq n$ for which the second case occurs by
$m(n) = m$. The set $\{m(n);\, n\in\N\}$ of all such numbers is unbounded.
Otherwise, assume that $m$ is the maximal element of this set. Then we get
$d(B) = \lim\limits_{n \to \infty} \frac{C(n)+B'(n)}n \leq \lim\limits_{n \to \infty}
\frac{C(n)+A'(n)+B'(m)-A'(m)}n = d(A)$, a contradiction.

Now let $\eps>0$ and $N_0$ be such that for $k\geq N_0$ the inequality
$\abs{\frac{B'(n)-A'(n)}n-L}\leq\eps$ holds. Since the set $\{m(n);\, n\in\N\}$
is unbounded, we can choose $n$ large enough to assure that
$n\geq m(n)\geq N_0$. Then we get
\begin{gather*}
\begin{aligned}
(L+\eps)m(n) &\geq B'(m(n)) - A'(m(n)),\\
(L-\eps)n &\leq B'(n) - A'(n).
\end{aligned}
\end{gather*}
Hence $n(L-\eps) \leq B'(n) - A'(n) \leq B'(m(n)) - A'(m(n)) \leq (L+\eps)m(n)$
and
\begin{gather*}
m(n) \geq n\frac{L-\eps}{L+\eps},\\
n-m(n) \leq \frac{2\eps n}{L+\eps} \leq \frac{2\eps}L n.
\end{gather*}
We have $A'(n)-D'(n)\leq A'(n)-D'(m(n)) = A'(n)-A'(m(n)) \leq n-m(n)$,
$$0\leq \frac{A'(n)-D'(n)}n \leq \frac{2\eps}L$$
and
$$\lim\limits_{n \to \infty} \frac{A'(n)-D'(n)}n =0.$$
\end{proof}

Of course the claim of this lemma holds also if $d(A)=d(B)$. We proved in
fact also the following result:

\begin{lem}\label{4:LMM1B}
If $A\cap B=\emptyset$, $\lim\limits_{n \to \infty} \frac{B(n)-A(n)}n$
exists and $\lim\limits_{n \to \infty} \frac{B(n)-A(n)}n
>0$, then there is a subset $D\subseteq B$ with $\lim\limits_{n \to
\infty} \frac{D(n)-A(n)}n=0$. In particular, $B\smallsetminus D\in\mcl D$.
\end{lem}

\begin{cor}\label{4:COR1}
If $A,B\in\mcl D$, $d(A)<d(B)$, then there exists $D\in\mcl D$ such that
$A\subseteq D\subseteq A\cup B$, $D\in\mcl D$ and $d(D)=d(B)$.
\end{cor}

\begin{proof}
We have $d(\N\smallsetminus A)>d(\N\smallsetminus B)$. By Lemma
\ref{4:LMM1} there exists a set $E\in\mcl D$ such that
$\N\smallsetminus(A\cup B)\subseteq E \subseteq \N\smallsetminus A$
and $d(E)=d(\N\smallsetminus B)$. If we put $D = \N\smallsetminus
E$, then $A\subseteq D \subseteq A\cup B$ and $d(D)=d(B)$.
\end{proof}

\begin{lem}\label{4:LMM2}
Let $A\subseteq\N$. Then there exists a subset $B\subseteq A$ such that
$B\in\mcl D$ and $d(B)=\lld(A)$.

Similarly, there exists a superset $C\supseteq B$ such that $C\in\mcl D$ and $d(C)=\uud(A)$.
\end{lem}

\begin{proof}
By the definition of $\lld(A)$ we have $\lld(A) = \sup\{d(B);\, B\subseteq
A,\, B\in \mcl{D}\}$. By results of Grekos and Volkmann \cite{grv} the density set $S(A)$ of all
density points $(\ld B,\ud B)$, $B\subseteq A$, is closed, hence it
contains its accumulation point $\big(\lld (A),\lld (A)\big)$. This point corresponds
to the desired subset $B$ of $A$.

The proof of the second part is analogous.
\end{proof}

\begin{lem}\label{4:LMM5}
If $A\cap B=\emptyset$, $A\in\mcl D$, $\lld(B)=0$, then $\lld(A\cup
B)=d(A)$.
\end{lem}

\begin{proof}
Assume that $\lld(A \cup B)>d(A)$. Then there is $C\subseteq A\cup
B$ with $d(C)>d(A)$. By Corollary \ref{4:COR1} we may assume that
$C\supseteq A$. Then $C\smallsetminus A\in\mcl D$,
$d(C\smallsetminus A)= d(C)-d(A)>0$ and therefore $\lld(B) >0$, a
contradiction.
\end{proof}

\begin{lem}\label{4:LMM6}
If $A\in\mcl D$, $A\cap B=\emptyset$, then $\lld(A\cup B)=d(A)+\lld (B)$.
\end{lem}

\begin{proof}
By Lemma \ref{4:LMM2} there exists $B_1\subseteq B$ such that
$d(B_1)=\lld B$. Clearly, $\lld(B\smallsetminus B_1)=0$. Then
using Lemma \ref{4:LMM5} we get $\lld(A\cup B)=\lld (A\cup
B_1\cup(B\smallsetminus B_1))= d(A\cup
B_1)=d(A)+d(B_1)=d(A)+\lld(B)$.
\end{proof}

\begin{lem}\label{4:LMM7}
If $B,C\in\mcl D$ and $A\cup B \supseteq C$, then $d(C)-d(B) \leq \lld(A)$.
\end{lem}

\begin{proof}
$d(C) \leq \lld(A\cup B)= \lld((A\smallsetminus B) \cup B) = \lld
(A\smallsetminus B)+d(B) \leq \lld(A) + d(B)$.
\end{proof}

We can see that the expression from Lemma \ref{4:LMM7} appears also in the definition of $d_*$
(it is equal to $\sum_{i=1}^{p}d(A_i) - \sum_{j=1}^{q}d(B_j)$
for a special case $p = q =1$). To prove Theorem \ref{lld=di} it suffices to show that every
such a difference of two sums can be transformed to this simple case.

Let $A = \{a_1 < a_2 < \ldots\}$ be infinite and $m \in \N$.
Define a set $B = \{b_1 < b_2 < \ldots\}$, where  $b_i$ is an
arbitrary number from the set $\{m a_i + 1, m a_i + 2, \dots, m
a_i + m\}$. We will call the set of this kind an {\em $m$-copy of
$A$\/}. Then it is easy to see that $\lld(A) = m\cdot \lld(B)$.
We have also $\lim\limits_{n \to \infty} \frac{A(n)}{B(n)} = m$, thus
$d(A) = m\cdot d(B)$ whenever $A \in \mcl{D}$. Let us note that by
\cite[Theorem 1]{st} $\mu(A)=m\cdot\mu(B)$ holds for any
density measure as well.

\begin{proof}[Proof of Theorem \ref{lld=di}]
Let
\begin{gather}\label{kchiA}
k \chi_A + \sum_{j=1}^{q}\chi_{B_j} \geq \sum_{i=1}^{p}\chi_{A_i}.
\end{gather}
Put $l = \sum_{j=1}^{q}\chi_{B_j}$ and  $r = \sum_{i=1}^{p}\chi_{A_i}$.
Let $m \geq \max\limits_{n \in \N}(k\chi_A(n) + l(n))$.
Taking
\begin{gather*}
\begin{aligned}
C &= \bigcup_{n\in\N} \{m n+1, m n +2, \ldots, mn+l(n)\}\\
D &= \bigcup_{n\in A} \{mn+l(n)+1, mn+l(n)+2, \ldots, mn+l(n)+k\},\\
E &= \bigcup_{n\in\N} \{m n+1, m n +2, \ldots, mn+r(n)\}
\end{aligned}
\end{gather*}
 we get from \refeq{kchiA}
$C\cup D\supseteq E$.

The sets $C$, $D$, $E$ can be viewed as a disjoint union of
$m$-copies of the sets $B_1, B_2, \dots, B_q$, a disjoint union of
$k$ $m$-copies of the set $A$ and a disjoint union of $m$-copies
of the sets $A_1, A_2, \dots, A_p$, respectively. Hence we have
$d(C) = \frac{1}{m}\sum\limits_{i=1}^q d(B_i)$, $\lld(D) =
\frac{k}{m} \cdot \lld(A)$
 and
$d(E) = \frac{1}{m}\sum\limits_{i=1}^p d(A_i)$.
Thus by Lemma \ref{4:LMM7}:
\begin{gather*}
d(E)-d(C)\leq \lld(D),\\
\frac{\sum_{i=1}^{n}d(A_i) - \sum_{j=1}^{m}d(B_j)}{m} \leq \frac{k}{m} \cdot \lld(A),\\
\frac{\sum_{i=1}^{n}d(A_i) - \sum_{j=1}^{m}d(B_j)}{k} \leq \lld(A).
\end{gather*}
Hence, $d_*(A) \leq \lld(A)$. From \refeq{lld<di} we have the reverse inequality, so we get
$d_*(A) = \lld(A)$. The dual equality $d^*(A) = \uud(A)$ follows from
$d_*(\N \smallsetminus A) = \lld(\N \smallsetminus A)$.
\end{proof}

The simplification obtained in Theorem \ref{lld=di} applied to the
results of \cite[Proposition 3.2.8]{brbr} yields the
following:

\begin{prop}
If $A,B\subset\N$ and $A\cap B=\emps$, then
$$\lld(A)+\lld(B) \leq \lld(A\cup B) \leq \lld(A)+\uud(B) \leq \uud(A\cup B) \leq \uud(A)+\uud(B).$$

If $A\cap B=\emps$ and $A\cup B\in\mcl D$, then
   $$d(A\cup B)=\lld(A)+\uud(B).$$

If $A\in\mcl D$, $A\cap B=\emptyset$, then $$\uud(A\cup B)=d(A)+\uud(B).$$
\end{prop}

It is easy to find examples showing that the above inequalities can be strict.

As an application of the lemmas used in the proof of Theorem \ref{lld=di} we prove some
other interesting properties of $\lld$ and $\uud$ and of density measures.

%

\begin{prop}
If $A\subseteq B$ and $B\in\mcl D$, then there exists $C\in\mcl D$ such that $d(C)=\uud(A)$
and $A\subseteq C\subseteq B$.

Similarly, if $A\subseteq B$ and $A\in\mcl D$, then there exists $C\in\mcl D$ such that
$d(C)=\lld(B)$ and $A\subseteq C\subseteq B$.
\end{prop}

\begin{proof}
By Lemma \ref{4:LMM2} there exists $D\in\mcl D$ such that $d(D)=\uud(A)$ and $A\subseteq D$. Clearly,
$d(D)=\uud(A)\leq \uud(B)=d(B)$. By Lemma \ref{4:LMM1} there exists $C$ such that
$A \subseteq D \cap B \subseteq C \subseteq B$ and
$d(C)=d(D)=\uud(A)$.

The second part is dual to the first one.
\end{proof}

\begin{lem}\label{4:LM2B}
If $A,B\subseteq\N$, and $\lim\limits_{n \to \infty} \frac{B(n)-A(n)}n$
exists and $\lim\limits_{n \to \infty} \frac{B(n)-A(n)}n
>0$, then there is a set $D$ such that $A\cap B\subseteq D\subseteq B$ with $\lim\limits_{n \to
\infty} \frac{D(n)-A(n)}n=0$. In particular, $B\smallsetminus D\in\mcl D$.
\end{lem}

\begin{proof}
Use Lemma \ref{4:LMM1B} for $B':=B\smallsetminus A\cap B$ and $A':=A\smallsetminus A\cap B$.
\end{proof}

\begin{prop}
Let $A,B\subseteq\N$. There exists the limit $\lim\limits_{n \to \infty} \frac{B(n)-A(n)}n=L$
\iaoi $\mu(B)-\mu(A)=L$ for every density measure $\mu$.
\end{prop}

\begin{proof}
If $\lim\limits_{n \to \infty} \frac{B(n)-A(n)}n=L$ then, by Lemma \ref{4:LM2B}, there exists
a subset $D\subset B$ with $\limti n \frac{D(n)-A(n)}n=0$. This implies $\mu(D)=\mu(A)$ for every
density measure $\mu$ by \cite[Proposition 3.3]{SLEZZIMDENSLEVY}. From this we get
$$\mu(B)=\mu(D)+\mu(B\sm D)=\mu(A)+d(B\sm D)=\mu(A)+L.$$

On the other hand, if $\mu(B)-\mu(A)=L$ for each density measure $\mu$, then also
$$\flim \frac{B(n)-A(n)}n=L$$
for every free ultrafilter $\mcl{F}$. This implies that the only cluster point of the sequence
$\left(\frac{B(n)-A(n)}n\right)$ is $L$ and
$$\limti n \frac{B(n)-A(n)}n=L.$$
\end{proof}

\begin{rem}\label{REMPOLYA}
The functions $\lld$ and $\uud$ were studied also by P\'olya
\cite{polya} in a~more general setting. He has studied sequences
of non-negative real numbers such that the difference of successive
elements is bounded from 0. We will use his result only for sequences
of natural numbers.
Among other things he proved in \cite[Satz VIII]{polya} that
\begin{gather*}
\lld(A) = \lim_{\theta \to 1^-} \liminf_{n \to \infty} \frac{A(n) - A(\theta n)}{n - \theta n},
\qquad
\uud(A) = \lim_{\theta \to 1^-} \limsup_{n \to \infty} \frac{A(n) - A(\theta n)}{n - \theta n}.
\end{gather*}
P\'olya called these values minimal and maximal density (Minimaldichte and Maximaldichte).
This expression of $\lld$ and $\uud$ can serve as a basis for a different proof of Corollary \ref{dmvalues}. 

Finitely additive measures on $\N$ can be understood as positive normed functionals on $\ell_\infty$.
Clearly, if we identify a subset of $\N$ with its characteristic sequence, such
a functional yields a measure on $\N$. The functional corresponding to a measure is in fact the integral
with respect to this measure (obtained by imitating the definition of Riemann integral,
see e.g.~\cite[Section 3]{vandouwen}). A more detailed exposition into representation of
finitely additive measures as the elements of the dual space $\ell_\infty^*$ can be found in
\cite{blumlinger}.
From the positivity and $\mu(\N)=1$ we see that the norm of each measure on $\N$ in $\ell_\infty^*$ is equal to 1.

Suppose we are given a set $A\subseteq\N$.
Now, for a given $\theta<1$, choose a sequence $(n_i)$ such that
$$\limti i \frac{A(n_i)-A(\theta n_i)}{n_i-\theta n_i} = \liminf_{n \to \infty} \frac{A(n) - A(\theta n)}{n - \theta n}.$$
Let $\FF$ be any free ultrafilter containing the set $\{n_i; i\in\N\}$. Then
$$\mu_\theta(B)=\flim \frac{B(n)-B(\theta n)}{n-\theta n} =
\flim \left(\frac{B(n)}n \frac n{n-\theta n}+\frac{B(\theta n)}{\theta{n}} \left(1-\frac n{n-\theta{n}}\right)\right)$$
defines a density measure on $\N$ such that $$\mu_\theta(A)=\liminf\limits_{n \to \infty} \frac{A(n) - A(\theta n)}{n - \theta n}.$$

We consider all these measures as the elements of the unit ball
of $\ell^*_\infty$. By Banach-Alaoglu theorem this ball is compact
in weak${}^*$ topology. Let us choose a sequence $\mu_{\theta_k}$
such that $\limti k\theta_k=1$ and $0<\theta_k<1$. Then the sequence
$(\mu_{\theta_k})$ has a subsequence which is convergent in the weak${}^*$
topology. We denote this subsequence by $(\mu_n)$ and the limit by $\mu$.
The convergence in weak${}^*$ topology implies that
$$\mu(A)=\limti n \mu_n(A) = \lld (A).$$

The proof that there exists a density measure with $\mu(A)=\uud(A)$ is similar.
Together with the convexity of the set of density measures and the obvious estimates
$\lld(A)\le\mu(A)\le\uud(A)$ this yields Corollary \ref{dmvalues}.
\end{rem}

\section{Density measures with a given value for some set}

Corollary \ref{dmvalues} gives the complete answer to the question
of Fey about the values of density measure. But this answer is
only existential. The natural question arisen here is: {\em Which
values are attained by density measures expressible in some
``simple'' form, e.g. \refeq{new repr}?}

\begin{lem}\label{meas lemma1}
Consider $\alpha \geq -1$ and $A \subseteq \N$. Then the set of all cluster points of the
sequence $\Big(\frac{A_\alpha(n)}{\N_\alpha(n)}\Big)$ is the whole interval
$\big[\lda{\alpha}(A), \uda{\alpha}(A)\big]$.
\end{lem}

\begin{proof}
The proof of this lemma is based on \cite[Theorem 1]{aa} which claims that if a~sequence
$(x_n)$ in a compact metric space $(X, d)$ satisfies
\begin{gather*}
\lim_{n \to \infty}d(x_n, x_{n+1}) = 0,
\end{gather*}
then the set of all cluster points of $(x_n)$ is connected.
The connectedness of the set of cluster points of $\Big(\frac{A_\alpha(n)}{\N_\alpha(n)}\Big)$
is equivalent to the assertion of the lemma.

As $0 \leq \frac{A_\alpha(n)}{\N_\alpha(n)} \leq 1$ for all $n \in \N$, it suffices to show that
$\lim\limits_{n \to \infty}\abs{\frac{A_\alpha(n)}{\N_\alpha(n)}-
\frac{A_\alpha(n+1)}{\N_\alpha(n+1)}}=0$.
Assume that $\alpha > -1$.
Then we have
\begin{gather*}
\begin{aligned}
\left| \frac{A_\alpha(n)}{\N_\alpha(n)} - \frac{A_\alpha(n+1)}{\N_\alpha(n+1)}\right|
&\leq \left| \frac{A_\alpha(n)}{\N_\alpha(n)} - \frac{A_\alpha(n)}{\N_\alpha(n+1)}\right|
+ \frac{(n+1)^\alpha}{\N_\alpha(n+1)}\\
&= \frac{A_\alpha(n)}{\N_\alpha(n)} \left| 1 - \frac{{\N_\alpha(n)}}{\N_\alpha(n+1)}\right|
+ \frac{(n+1)^\alpha}{(n+1)^{\alpha+1}}\, \frac{(n+1)^{\alpha+1}}{\N_\alpha(n+1)}
\end{aligned}
\end{gather*}
Now, using Lemma \ref{alpha int}, it can be easily seen that
$\lim\limits_{n \to \infty}
\abs{\frac{A_\alpha(n)}{\N_\alpha(n)} - \frac{A_\alpha(n+1)}{\N_\alpha(n+1)}} = 0$.

Analogous treatment can be used also for $\alpha = -1$. The only
difference is replacing the term $(n+1)^{\alpha+1}$ by $\ln(n+1)$
in the last part of the above estimation.
\end{proof}

\begin{cor}\label{meas lemma2}
Let $A \subseteq \N$ and $\alpha \geq -1$. For every
$x \in \big[\lda{\alpha}(A), \uda{\alpha}(A)\big]$ there is a~free ultrafilter $\FF$ such that
$\mu_\alpha^\FF(A) = x$.
\end{cor}

\begin{proof}
According to Lemma \ref{meas lemma1}, $x$ is a cluster point of the
sequence $\Big(\frac{A_\alpha(n)}{\N_\alpha(n)}\Big)$. Hence there
is an (infinite) set $K = \{n_1 < n_2 < \dots \}$ such that $x =
\lim\limits_{k \to \infty} \frac{A_\alpha(n_k)}{\N_\alpha(n_k)}$.
Taking any free ultrafilter $\FF$ containing the set $K$ one can
easily show that $\mu_{\alpha}^{\FF}(A) = x$. This completes the
proof.
\end{proof}

The following result follows from Rajagopal \cite{rajagopal}.

\begin{thm}\label{alpha < beta}
Let $-1 \leq \alpha \leq \beta$. Then for all $A \subseteq \N$ we have
\begin{gather*}
\lda{\beta}(A) \leq \lda{\alpha}(A) \leq \uda{\alpha}(A) \leq \uda{\beta}(A).
\end{gather*}

\end{thm}

This led us to introduce the following notation:
\begin{gather*}
\begin{aligned}
\lda{\infty}(A) &= \lim\limits_{\alpha \to \infty} \lda{\alpha}(A) =
\inf\limits_{\alpha \geq -1} \lda{\alpha}(A);\\
\uda{\infty}(A) &= \lim\limits_{\alpha \to \infty} \uda{\alpha}(A) =
\sup\limits_{\alpha \geq -1} \uda{\alpha}(A).
\end{aligned}
\end{gather*}

\begin{thm}
If $A \subseteq \N$ and $x \in \big(\lda{\infty}(A), \uda{\infty}(A)\big)$, then there is
a density measure $\mu$ of the form \refeq{new repr} such that $\mu(A) = x$.
\end{thm}

\begin{proof}
By definition of $\lda{\infty}(A)$ and $\uda{\infty}(A)$ there is
an $\alpha \geq -1$ such that $x \in \big[\lda{\alpha}(A),
\uda{\alpha}(A)\big]$. The rest follows from Corollary \ref{meas
lemma2}.
\end{proof}

Using Theorem \ref{first_estim} we get

\begin{cor}\label{lld < ld infty}
For all $A \subseteq \N$ we have
\begin{gather*}
\lld(A) \leq \lda{\infty}(A) \leq \uda{\infty}(A) \leq \uud(A).
\end{gather*}
\end{cor}

Using Corollary \ref{dmvalues} and Corollary \ref{lld < ld infty}
we are able to show that the expression \refeq{new repr} does not
describe all density measures. There are also density measures of
different type:

Again, taking the set $A = \bigcup\limits_{k=0}^\infty
\big(2^{2k}, 2^{2k + 1}\big] \cap \N$ we have $\lld(A) =
\lda{\infty}(A) = 0$. On the other hand, $\mu_\alpha^\FF(A) \geq
\frac{1}{2^{\alpha+1}+1} > 0$. Hence $\mu(A) = \int_{\Omega}
\mu_\alpha^\FF(A)\, \dd\psi(\FF, \alpha)> 0$ for any probability
Borel measure $\psi$, too. But by Corollary \ref{dmvalues} there
is also a density measure $\mu'$ with $\mu'(A) = \lld(A) = 0$.

We recall the definition of gap density and some results from
\cite{grv}. The value of the gap density $\lambda(A)$ describes
how large gaps can be between elements of $A$. It is given by
$$\lambda(A) = \limsup\limits_{n\to\infty} \frac{a_{n+1}}{a_n}$$
for $A=\{a_1<a_2<a_3<\ldots\}$. The sets having infinite
gap density are called thin sets in \cite{bfpr}.

It is shown in \cite{grv} that the density set
$S(A)=\{(\ld(B),\ud(B)); B\subseteq A\}$ is located above the line
$y=\lambda(A)x$. It follows, that if $\lambda(A)>1$, then
$\ud(B)\geq\lambda(A)\ld(B)>\ld(B)$ for any subset $B$ of $A$ with
$\ld(B)>0$ (see also \cite[Proposition 2.1]{gst}). Hence no subset
of $A$ has density strictly greater than 0 and $\lld(A)=0$.

On the other hand, if $\lambda(A)>1$, then there are arbitrary large $n$'s with
$a_{n+1}>(1+\eps)a_n+1$.
Thus if $b_n=\uip{(1+\eps)a_n}$ we get $A_\alpha(b_n)=A_\alpha(a_n)$ and
$b_n\geq (1+\eps)a_n$. Hence
$(\alpha+1)\frac{A_\alpha(b_n)}{b_n^{\alpha+1}} =
(\alpha+1)\frac{A_\alpha(a_n)}{b_n^{\alpha+1}}\leq
\frac1{(1+\eps)^{\alpha+1}}\frac{(\alpha+1)\N_\alpha(a_n)}{{a_n}^{\alpha+1}}$ and
by Lemma \ref{alpha int} we get
$\lda{\alpha}(A)\leq\frac1{(1+\eps)^{\alpha+1}}$. Consequently $\lda{\infty} (A)=0$.

\begin{prop}
If $\lambda(A)>1$, then $\lld(A) = \lda{\infty}(A)=0$.
\end{prop}

The question of the equality of $\lld(A)$
and $\lda{\infty}(A)$ for a set $A$ with $\lambda(A) = 1$ remains open.

\begin{prob}
Is it true that $\lld(A) = \lda{\infty}(A)$ for every $A \subseteq \N$?
\end{prob}

\noindent\textbf{Acknowledgement:} Authors are grateful to Professor Georges Grekos for
helpful discussion on density sets. We are also greatly indebted to an anonymous referee
for pointing out that our Corollary \ref{dmvalues} can be deduced from Polya's results
(see Remark \ref{REMPOLYA}).


\begin{thebibliography}{10}
\expandafter\ifx\csname url\endcsname\relax
  \def\url#1{\texttt{#1}}\fi
\expandafter\ifx\csname urlprefix\endcsname\relax\def\urlprefix{URL }\fi

\bibitem{aa}
M.~D. A{\v s}i{\'c}, D.~D. Adamovi{\'c}, Limits points of sequences in metric
  spaces, Amer. Math. Monthly 77 (1970) 613--616.

\bibitem{balste}
B.~Balcar, P.~{\v{S}}t\v{e}p\'anek, Teorie mno\v{z}in, Academia, Praha, 1986 (in Czech).

\bibitem{BANACHLINEN}
S.~Banach, Theory of Linear Operations, North Holland, Amsterdam, 1987.

\bibitem{brbr}
K.~P.~S. Bhaskara~Rao, M.~Bhaskara~Rao, Theory of Charges -- {A} Study of
  Finitely Additive Measures, Academic Press, London--New York, 1983.

\bibitem{bfpr}
A.~Blass, R.~Frankiewicz, G.~Plebanek, C.~Ryll-Nardzewski, A note on extensions
  of asymptotic density, Proc. Amer. Math. Soc. 129~(11) (2001) 3313--3320.

\bibitem{blumlinger}
M.~Bl\"{u}mlinger, L\'{e}vy group action and invariant measures on {$\beta
  \mathbb{N}$}, Trans. Amer. Math. Soc. 348~(12) (1996) 5087--5111.

\bibitem{BUCK}
R.~C. Buck, The measure theoretic approach to density, Amer. J. Math. 68 (1946)
  560--580.

\bibitem{CAMPBELLKELLYASYMP}
D.~E. Campbell, J.~S. Kelly, Asymptotic density and social choice trade-offs,
  Math. Social Sci. 29 (1995) 181--194.

\bibitem{fey}
M.~Fey, May's theorem with an infinite population, Social Choice and Welfare 23
  (2004) 275--293.

\bibitem{fey2}
M.~Fey, Problems {(Density measures)}, Tatra Mnt. Math. Publ. 31 (2005)
  177--181.

\bibitem{fga}
A.~Fuchs, R.~Giuliano~Antonini, Th\'{e}orie g\'{e}n\'{e}rale des densit\'{e}s,
  Rend. Acc. Naz. delle Scienze detta dei XL, Mem. di Mat. 108 (1990) Vol. XI,
  fasc. 14, 253--294.

\bibitem{gjrcf}
L.~Gillman, M.~Jerison, Rings of Continuous Functions, Van Nostrand, Princeton,
  1960.

\bibitem{ggm}
R.~Giuliano~Antonini, G.~Grekos, L.~Mi\v{s}\'\i{}k, On weighted densities,
  Czechosl. Math. J. 57~(3) (2007) 947--962.

\bibitem{GREDENSSURV}
G.~Grekos, On various definitions of density (survey), Tatra Mt. Math. Publ. 31
  (2005) 17--27.

\bibitem{gst}
G.~Grekos, T.~{\v{S}}al{\'a}t, J.~Tomanov{\'a}, Gaps and densities, Bull. Math.
  Soc. Sci. Math. Roum. 46~(3-4) (2003-2004) 121--141.

\bibitem{grv}
G.~Grekos, B.~Volkmann, On densities and gaps, J. Number Theory 26 (1987)
  129--148.

\bibitem{HALBERSTAMROTH}
H.~Halberstam, K.~F. Roth, Sequences, Springer-Verlag, New York, 1983.

\bibitem{HOWARDRUBIN}
P.~Howard, J.~E. Rubin, Consequences of the axiom of choice, {Mathematical
  Surveys and Monographs. 59. Providence, RI: American Mathematical Society
  (AMS)}, 1998.

\bibitem{hrjech}
K.~Hrbacek, T.~Jech, {Introduction to set theory}, {Marcel Dekker}, New York,
  1999.

\bibitem{johnson}
B.~E. Johnson, Separate continuity and measurability, Proc. Amer. Math. Soc
  20~(2) (1969) 420--422.

\bibitem{lauwers}
L.~Lauwers, Intertemporal objective functions: strong {Pareto} versus
  anonymity, Mathematical Social Sciences 35 (1998) 37--55.

\bibitem{maharam}
D.~Maharam, {Finitely additive measures on the integers}, Sankhya, Ser. A 38
  (1976) 44--59.

\bibitem{PINCUSSOLOVAY1977}
D.~Pincus, R.~M. Solovay, Definability of measures and ultrafilters 42~(2)
  (1977) 179--190.

\bibitem{polya}
G.~P{\'o}lya, {Untersuchungen {\"u}ber L{\"u}cken und Singularit{\"a}ten von
  Potenzreihen.}, Math. Zeit. 29 (1929) 549--640.

\bibitem{rajagopal}
C.~T. Rajagopal, Some limit theorems, Amer. J. Math. 70~(1) (1948) 157--166.

\bibitem{st}
T.~{\v{S}}al\'{a}t, R.~Tijdeman, Asymptotic densities of sets of positive
  integers, Mathematica Slovaca 33 (1983) 199--207.

\bibitem{SLEZZIMDENSLEVY}
M.~Sleziak, M.~Ziman, {L}\'evy group and density measures, J. Number Theory,
128~(12) (2008) 3005--3012.

\bibitem{ten}
G.~Tenenbaum, {Introduction to analytic and probabilistic number theory},
  {Cambridge Univ. Press}, Cambridge, 1995.

\bibitem{TOMATRADEOFF}
V.~Toma, Densities and social choice trade-offs, Tatra Mt. Math. Publ. 31
  (2005) 55--63.

\bibitem{vandouwen}
E.~K. van Douwen, Finitely additive measures on {$\mathbb{N}$}, Topology and
  its Applications 47 (1992) 223--268.

\bibitem{walker}
R.~C. Walker, The {S}tone-\v{C}ech compactification, Springer-Verlag, Berlin,
  Heidelberg, New York, 1974.

\end{thebibliography}

\end{document}